\documentclass[a4paper]{amsart}
\usepackage{amsmath,amssymb,enumerate,bbm}

\newcommand{\tmem}[1]{{\em #1\/}}
\newcommand{\tmop}[1]{\ensuremath{\operatorname{#1}}}

\newenvironment{enumeratenumeric}{\begin{enumerate}[1.] }{\end{enumerate}}
\newenvironment{enumerateroman}{\begin{enumerate}[i.] }{\end{enumerate}}

\usepackage{amsfonts}
\usepackage{amsmath}
\usepackage{amssymb}
\usepackage{color}
\usepackage{graphicx}
\usepackage{xspace}
\usepackage[matrix,arrow]{xy}

\usepackage{mathrsfs,eucal,enumerate}
\usepackage{amsthm,amscd}
\usepackage{mathtools}
\usepackage{epsfig,psfrag}

\usepackage{shuffle}

\newtheorem{prop}{Proposition}
\newtheorem{lem}{Lemma}

\newtheorem{thm}{Theorem}
\newtheorem{de}{Definition}

\DeclareMathOperator{\Res}{Res}    

\newcommand{\eps}{\varepsilon}     
\newcommand{\CC}{\mathbb{C}} 
  
\newcommand{\hL}{L} 
\newcommand{\hGL}{G} 

\newcommand{\id}{\mathrm{id}}

\newcommand{\al}{\alpha}

\newcommand{\del}{\delta}

\newcommand{\dsp}{\displaystyle}

\begin{document}

\title{From dynamical systems to renormalization}\author{Fr\'ed\'eric Menous}

\address{D\'epartement de Math\'ematiques, B\^at. 425, UMR 8628, CNRS, Universit\'e Paris-Sud, 91405 Orsay Cedex, France }
\begin{abstract}
We study in this paper logarithmic derivatives associated to derivations on graded complete Lie algebra, as well as the existence of inverses. These logarithmic derivatives, when invertible, generalize the $\exp$--$\log$ correspondence between a Lie algebra and its Lie group. Such correspondences occur naturally in the study of dynamical systems when dealing with the linearization of vector fields and the non--linearizability of a resonant vector fields corresponds to the non--invertibility of a logarithmic derivative and to the existence of normal forms. These concepts, stemming from the theory of dynamical systems, can be rephrased in the abstract setting of Lie algebra and the same difficulties as in perturbative quantum field theory (pQFT) arise here. Surprisingly, one can adopt the same ideas as in pQFT with fruitful results such as new constructions of normal forms with the help of the Birkhoff decomposition. The analogy goes even further (locality of counter terms, choice of a renormalization scheme) and shall lead to more interactions between dynamical systems and quantum field theory.   
\end{abstract}

\maketitle


\section{Introduction.}

Since the work of A. Connes and D.K Kreimer (see \cite{ck1}, \cite{ck2}) in perturbative quantum field theory (pQFT), it  has been possible to have a purely algebraic interpretation of some renormalization schemes, as the Birkhoff decomposition of regularized characters (with values in a commutative algebra of Laurent series), that is of elements of the group of algebra morphisms from a graded commutative Hopf algebra to the the algebra of Laurent series (see \cite{Man}, \cite{CM}). Many graded commutative Hopf algebras (shuffle, quasi shuffle, tree) do appear in this framework and, surprisingly, the same objects (groups of characters on Hopf algebras) arise in the study of analytic dynamical systems (see for example \cite{SNAG}), where changes of coordinates can be computed as elements either of the group of formal identity-tangent diffeomorphisms (namely characters over the Fa\`a di Bruno Hopf algebra) or of some subgroup that corresponds to characters over a combinatorial Hopf algebra of trees or words (using mould calculus, see \cite{SNAG}, \cite {FM}).

Such interactions should allow to enrich each domain with the ideas of the other and some steps have already been done in this direction (see \cite{MenDE} or \cite{MenBD}).

The origin of this paper comes from the observation that, in pQFT, the need for renormalization comes from the ill-definedness of some character on a Hopf algebra, which, after dimensional regularization, is replaced by a character with values in an algebra of Laurent series. The attempted character is then obtained with the help of the Birkhoff decomposition. The same phenomenon occur in dynamical systems, in the study of normal forms of vector fields (when the the change of coordinates that should linearize the vector field is ill-defined). Apart from some remarks in the literature (see for example \cite{Gal}), there does not seem that renormalization schemes have been used to compute, for example, normal forms.

We shall explain in this paper in an abstract  algebraic context -- a common framework to pQFT and dynamical systems -- how the ideas developed in pQFT (dimensional regularization, Birkhoff decomposition ...) provide a way to compute normal forms as well as other objects that play a crude role in dynamics. 

In the next section, we briefly recall the definition of complete graded Lie algebras and their link with commutative Hopf algebras. We define then in section \ref{secder} graded derivations on Lie algebras and their associated logarithmic derivative. Under some invertibilty condition on such derivations, we define the inverse of such logarithmic derivative (see also section \ref{secovin}) that generalizes the $\log$--$\exp$ correspondence between a Lie algebra and its Lie group.  Whenever such a derivation is not invertible, the theory of dynamical systems (see appendix \ref{secds}), provide, in the framework of Lie algebra, a way to study these logarithmic derivatives, using $d$-conjugacy and "normal" forms (see section \ref{secnoninv}). There are surprising similarities with the  need for renormalization in pQFT, the same ideas --regularization and Birkhoff decomposition -- provide a new method to compute normal forms (section \ref{norenor}) and the analogy goes even further, see section \ref{secmore}, since such notions as the locality of counter terms or the choice among several "renormalization schemes" also appear here. 

\section{Lie and Hopf algebras.}

In the sequel we consider complete graded Lie algebras $L^k$ (as well as
complete vector spaces or algebras) over a field $k$ of characteristic 0. In
other words, $L^k$ is the completion $\prod_{n \geqslant 1} L^k_n$ of a graded Lie
algebra $\bigoplus_{n \geqslant 1} L^k_n$ (with finite dimensional graded components $L^k_n$ and $[L^k_n, L^k_m] \subset L^k_{n + m}$)
with respect to the graduation. We write for short $x = \sum_{n \geqslant 1}
x_n$ ($x_n \in L^k_n$) and the Lie bracket extends to the completion $L^k$: if $x =
\sum_{n \geqslant 1} x_n$ and $y = \sum_{n \geqslant 1} y_n$ then $[x, y] =
\sum_{n \geqslant 1} z_n$ with
\[ z_n = \sum_{k = 1}^{n - 1} [x_k, y_{n - k}] \]
Under these hypotheses, the universal enveloping algebra $\mathcal{U}(L^k)$
of $L^k$ is a unital, graded associative algebra with a graduation $0$ component
$k. 1$. This algebra turns into a graded co-commutative Hopf algebra for the graded
co-product defined on $L^k$ by
\[ \forall x \in L^k, \hspace{1em} \Delta (x) = 1 \otimes x + x \otimes 1 \]
that extends, thanks to the universal property of Lie algebras, to
$\mathcal{U}(L^k)$ (see \cite{bour}). Using once again the graduation, we note
$\mathcal{U}^k= \prod_{n \geqslant 0} \mathcal{U}^k_n$ the completion of
$\mathcal{U}(L^k)$ and, as a graded map the co-product extends to a map from
$\mathcal{U}^k$ to the completion of $\mathcal{U}(L^k) \otimes \mathcal{U}(L^k)$
(that may only contain $\mathcal{U}^k \otimes \mathcal{U}^k$). $\mathcal{U}^k$ is
not necessarily a Hopf algebra but it contains $L^k$ as the Lie algebra of \textit{primitive elements} of $\mathcal{U}^k$, that is the elements $x$ of
$\mathcal{U}^k$ such that $\Delta (x) = 1 \otimes x + x \otimes 1$, but also the
Lie group $G^k$ of group-like elements of $\mathcal{U}^k$, that is the elements
$x$ of $\mathcal{U}^k$ such that $\Delta (x) = x \otimes x$.

Let us also remind that there is a one to one correspondence between $L^k$
and $G^k$, using the well-defined exponential an logarithm:
\[ \begin{array}{ccccc}
     \exp & : & \mathcal{U}^k_{\geqslant 1} & \rightarrow & 1
     +\mathcal{U}^k_{\geqslant 1}\\
     &  & x & \mapsto & 1 + \sum_{n \geqslant 1} \frac{x^n}{n!}
   \end{array} \hspace{2em} \begin{array}{ccccc}
     \log & : & 1 +\mathcal{U}^k_{\geqslant 1} & \rightarrow &
     \mathcal{U}^k_{\geqslant 1}\\
     &  & 1 + y & \mapsto & \sum_{n \geqslant 1} (- 1)^{n - 1} \frac{y^n}{n}
   \end{array} \]
where $\mathcal{U}^k_{\geqslant 1} = \prod_{n \geqslant 1} \mathcal{U}^k_n$.

For a given commutative $k$--algebra $\mathcal{A}$, one can also consider the Lie algebra $L^{\mathcal{A}}$, which inherits grading of $L^k$ and is defined as the completion, for the same graduation as  in $L^k$, of the Lie algebra $\bigoplus_{n\geqslant 1} \mathcal{A}\otimes L^k_n$ with the Lie bracket
\[
[a\otimes x, b\otimes y] =(ab)\otimes[x,y]
\]
Note that $L^{\mathcal{A}}$ is both a vector space over $k$ (whose graded component may be of infinite dimension) and a module over $\mathcal{A}$ (with graded components of finite dimension) 
and one can as well define $\mathcal{U}^{\mathcal{A}}$ and  $G^{\mathcal{A}}$. 

There is no difficulty to extend iteratively this construction. For example, if $k\subset \mathcal{A} \subset \mathcal{B}$, then $L^k\subset L^\mathcal{A} \subset L^\mathcal{B}$ and the same inclusions hold for the completed enveloping algebras as well as for the groups. On the same way, we could define $\left(L^\mathcal{A} \right)^\mathcal{B}$ as the completion of $\bigoplus_{n\geq 1} \mathcal{B}\otimes  L^{\mathcal{A}}_n$. 
In the sequel, we will adopt this Lie algebra point of view but there is no difficulty to identify such Lie algebra $L^\mathcal{A}$ (resp. Lie groups $G^\mathcal{A}$) with the Lie algebra (resp. group) of infinitesimal characters (resp. characters) on a graded commutative algebra, with values in $\mathcal{A}$. The reader can refer to \cite{EFGBP} on this matter. This simply means that the ideas stemming from the study of characters on  Hopf algebras can be easily translated in the context of such Lie algebras. 

Let us give, as an example, a "toy model" stemming from the theory of dynamical systems. When dealing with the simple differential equation:
\[
y'=a(x)y^2 \quad \text{with } a(x)\in\CC\{x\}
\]
the Lie algebra (over $\CC$) of differential operator (or vector fields), 
\begin{equation}
L=\bigoplus_{n\geqslant 1} \text{Vect}_{\CC } \{x^{n-1} y^2 \partial_y\}
\end{equation}
appears with a natural action on series of $\CC[[x,y]]$ and the elements of the completion of $L$ are formal operators $a(x)y^2 \partial_y$ with $a(x)\in \CC[[x]]$. For short we write $y^2 \CC[[x]]\partial_y$ this Lie algebra. Note that it coincides, for a different graduation, with the Lie algebra $L^{\mathcal A}$ with $\mathcal{A}=\CC[[x]]$ and $L^{\CC}=L_1^{\CC}=\text{Vect}_{\CC} \{ y^2 \partial_y \}$.

As another example related to dynamical systems, consider the Lie algebra $L^{k[t]}$, whose elements can be written:
\[
x=\sum_{n\geq 1}\sum_{l=0}^{N_n} t^l x_{n,l} \text{ with } N_n\in\mathbb{N}, x_{n,l}\in L^k_n.
\]

The element $\varphi^t = \exp (t x) \in G ^{k[t]}$ ($x \in L^k$) is the
unique solution of the differential equation
\[ \partial_t \varphi^t = \varphi^t .x \]
with the initial condition $\varphi^0 = 1$. This suggests to explore similar
"differential" equations that may lead to new correspondences between Lie algebras and Lie groups, since such equations appear naturally in the framework of differential equations
(see section \ref{secds}).

In the case of our toy model, the differential operator $\exp(a(x)y^2\partial_y)$ is a substitution automorphism:
\[\forall f(x,y)\in \CC[[x,y]],\quad \exp(a(x)y^2\partial_y) .f(x,y)=f(x,\exp(a(x)y^2\partial_y).y)
\]
with $\exp(a(x)y^2\partial_y).y=\frac{y}{1-a(x)y}$ and the map $\exp(a(x)y^2\partial_y)\mapsto \frac{y}{1-a(x)y}$ is an anti--isomorphism between the Lie group of $ y^2 \CC[[x]]\partial_y$ and the group (for the composition) of formal identity-tangent diffeomorphisms
\[ \left \{(x,y)\mapsto \left( x,\frac{y}{1-a(x)y}\right) \right\}.
\] 
The reader can check that $\exp (ta(x)y^2\partial_y)=\frac{y}{1-ta(x)y}=\alpha^t(x,y)$ is the flow of the differential equation $y'=a(x)y^2$ (see section \ref{secds} for details).

For short, in the sequel $L$ (resp. $\mathcal{U}$, $G$) will denote $L^\mathcal{A}$ (resp. $\mathcal{U}^\mathcal{A}$, $G^\mathcal{A}$). Note also that any $k$--linear endomorphism $u$ of $\mathcal{B}$  (resp. graded $k$--linear endomorphism of $L$) can be extended to a $k$--linear graded endomorphism of $L^\mathcal{B}$, still noted $u$ when there is no ambiguity, setting $u(b\otimes x)=u(b)\otimes x$ (resp. $u(b\otimes x)=b\otimes u(x)$). In order to avoid ambiguities, whenever such an endomorphism is defined on possibly many $k$--vector spaces (or algebras) $V$, we will write the image $u(V)$ and the corresponding kernel
\[
\ker_V u=\ker u \cap V \subset V.
\]

\section{Derivations and $d$-logarithms in Lie algebras.}\label{secder}

\subsection{Definitions and first properties.}

\begin{de}
  A derivation $d$ is a graded $k$--linear endomorphism on $L$ ($L^k$ or $L^\mathcal{A}$) such that
  \[ \forall x, y \in L, \hspace{1em} d ([x, y]) = [d (x), y] + [x, d (y)] \]
\end{de}

From the universal property of Lie algebras, as $d$ is graded
($d (L_n) \subset L_n$) it extends to a graded
derivation on $\mathcal{U}$:
\begin{enumerateroman}
  \item $d (1) = 0$,
  
  \item $d (\mathcal{U}_n) \subset \mathcal{U}_n$ and
  
  \item $\forall x, y \in \mathcal{U}, \hspace{1em} d (x.y) = d (x) .y + x.d
  (y)$.
\end{enumerateroman}

For such a derivation, we can define and study the ``differential'' equation
\begin{equation}
  d (\varphi) = \varphi .x
\end{equation}
\begin{prop}
  Any such derivation defines a map:
  \begin{equation}
    \begin{array}{lllll}
      \log_d & : & 1 +\mathcal{U}_{\geqslant 1} & \rightarrow &
      \mathcal{U}_{\geqslant 1}\\
      &  & \varphi & \mapsto & \varphi^{- 1} .d (\varphi)
    \end{array}
  \end{equation}
  that sends $G$ on $L$.
\end{prop}

The map is well defined since any $\varphi = 1 + \sum_{n \geq 1} \varphi_n \in 1 +\mathcal{U}_{\geqslant 1}$
is invertible for the product in $\mathcal{U}$ with an inverse
\[ \varphi^{- 1} = 1 + \sum_{s \geq 1} (- 1)^s \sum_{n_1, ..., n_s \geq 1}
   \varphi_{n_1} ... \varphi_{n_s} \]
and, thanks to the completion, $\log_d$ is well-defined in the complete
algebra $\mathcal{U}$. It remains to prove that it maps $G$ on $L$ and this
results comes from the Magnus formula (see \cite{KMP}):

\begin{lem}\label{dexp}
  Let $\varphi = \exp (\alpha) \in G$ ($\alpha \in L$), then
  \begin{equation}
    \log_d (\varphi) = \sum_{s \geqslant 0} \frac{(- 1)^s}{(s + 1) !}
    \tmop{ad}^s_{\alpha} (d (\alpha)) = \frac{e^{- \tmop{ad}_{\alpha}} - 1}{-
    \tmop{ad}_{\alpha}} (d (\alpha))\label{adj}
  \end{equation}
  where $\tmop{ad}_{\alpha} (x) = [\alpha, x] = \alpha x - x \alpha$ is the
  adjoint action of $\alpha$.
\end{lem}

Since $\mathcal{U}$ is complete, the above
series converges in  $\mathcal{U}_{\geqslant 1}$ when $\alpha$ is in
$\mathcal{U}_{\geqslant 1}$. Moreover, since $\alpha$ is in $L$ ($\varphi \in G$), 
$d (\alpha)$ and $\tmop{ad}^s_{\alpha} (d (\alpha))$ belong to $L$ thus $\log_d
(\varphi)$ is in $L$.

Note that we would have similar formulas for $(d (\varphi)).\varphi^{- 1} $ ($ (d (\varphi)) .\varphi^{- 1} = \varphi . \log_d (\varphi) . \varphi^{-1}$) and, since $\varphi^{- 1} . \varphi = 1$, 
\[ 
\left \lbrace
\begin{array}{l}
\log_d (\varphi^{- 1}) = \varphi .d (\varphi^{- 1}) = -  d
   (\varphi) .\varphi^{- 1}= - \varphi . \log_d (\varphi) . \varphi^{-1} \\
    \\
    \log_d (\varphi_1 . \varphi_2) = \log_d (\varphi_2) + \varphi_2^{-1} \log_d
   (\varphi_1) \varphi_2\end{array} \right. . \]
This logarithmic derivative is a kind of generalization of the
logarithm and such $d$-logarithm do appear in the study of dynamical systems (see section \ref{secds}).

For a given $d$-logarithm from $G$ to $L$ it is natural to ask if there
exists an inverse.

\subsection{The invertible case.}
Whenever the derivation $d$ has a graded right inverse on $L$ ($L^k$ or $L^{\mathcal{A}}$), there is a positive answer to the previous question: 
\begin{thm} \label{rightinv}
If $d$ admits a graded right inverse $I$ on $\hL$ then for any $u \in \hL$ there exists $\varphi \in  \hGL$
such that $\log_d(\varphi)=u$.
\end{thm}

\begin{proof}
Let  $u \in \hL$. If we look for $\varphi=\exp(v)$
and set, thanks to the graduation, \[
v=\sum_{n\geq 1} v_n,\quad v_n\in L_n.
\]
Thanks to lemma \ref{dexp},
\begin{equation}\label{adjgrad}
d(v_n)+\sum_{k=1}^{n-1}\sum_{i=1}^{n-k}\frac{(-1)^i}{(i+1)!}\sum_{n_1+..+n_i=n-k}ad_{v_{n_1}}(
ad_{v_{n_2}}...(ad_{v_{n_i}}(d(v_{k})))...)=u_n
\end{equation}
In $L_1$ we have,
\[ d(v_1)=u_1 \] and we define $v_1=I(u_1)\in \hL$. For $n=2$,
\[
d(v_2)-\frac{1}{2} [v_1,d(v_1)]=u_2
\]
and we have the solution
\[
v_2=I(u_2)-\frac{1}{2} I([I(u_1),u_1])
\]
and the proof follows recursively: $v=\sum_{n\geq 1} v_n$ is
in $L$ and $\varphi=\exp(v)$ is a solution of $\log_d(\varphi)=u\in
\hGL$.
\end{proof}
If $d$ is invertible on $\hL$ then the solution $\varphi$ of $\log_d \varphi=u$ is
unique and noted
\[ \varphi =\exp_d (u). \]
The reader can notice that in the case of $L^k$, since we assumed that the graded components $L^k_n$ are of finite dimension, a derivation has a right inverse if and only if it is invertible. Otherwise,
\begin{prop} For any derivation $d$, $\ker_L d$ is a sub-Lie algebra
  of $L$ ($\{ 0\}$ when $d$ is invertible) that gives rise to a
  subgroup of $G$:
\[ \hGL^0=\{\varphi \in \hGL \mid d(\varphi)=0\}
\]
and $\varphi_1$ and $\varphi_2$ are two solutions of $\log_d \varphi =u$ if and only if
$\varphi_2\varphi_1^{-1}\in  \hGL^0$
\end{prop} 
\begin{proof}
Consider $\psi=\varphi_2\varphi_1^{-1}$, that is $\psi\varphi_1=\varphi_2$,
then
\[ \log_d(\psi\varphi_1 )= \log_d(\varphi_1)+\varphi_1^{-1} \log_d(
\psi)\varphi_1= u+\varphi_1^{-1}  \log_d(
\psi)\varphi_1=u
\]
thus
\[\psi^{-1} d(\psi) =\log_d(
\psi)=0\]
and $d(\psi)=0$.
\end{proof}
 Note that, using once again lemma \ref{dexp}, we get:
\begin{equation}
 \hGL^0=\{\varphi=\exp(v) \in \hGL \mid v\in \ker_L d\}=\exp(\ker_L d)
\end{equation}

\subsection{The Dynkin operator, the logarithm and other examples.}

A classical (and universal) example of invertible derivation in
graded Lie algebras is the derivation associated to the graduation:
\begin{equation}
\forall n\geq 1,\ \forall x\in L_n,\quad Y(x)=nx. 
\end{equation}
In this case, $\log_Y$ is the Dynkin operator, which is a
bijection between $L$ and $G$ (see \cite{EFGBP}). 

Let us just point out another
example that relates the Dynkin operator and the usual logarithm. For a given complete graded Lie algebra $L=L^\mathcal{A}$ over a commutative $k$--algebra $\mathcal{A}$,  consider the Lie algebra $L^{\mu \mathcal{A}[\mu]}$. This algebra inherits the graduation of $L$ and the graded derivation $\mu \partial_{\mu}$ is invertible. If $x=\sum x_n\in L$, then $\mu x$ and $\mu^Y(x)=\sum
\mu^n x_n$ are in $L^{\mu \mathcal{A}[\mu]}$ and:
\begin{eqnarray*}
\exp_{\mu \partial_{\mu}}(\mu x)&=&\exp(\mu x)\\
\exp_{\mu \partial_{\mu}}(\mu^Y(x))&=&\mu^Y(\exp_Y(x))
\end{eqnarray*}
so that the classical exponential as well as $\exp_Y$
are related to the same differential equation in $L^{\mu \mathcal{A}[\mu]}$:
\[ \mu \partial_{\mu}\varphi=\varphi . x\]

We will come back to the interaction between the logarithm and the
Dynkin operator in the following section but let us end by the
definition of a wide class of derivations that appear in dynamical
systems. Suppose that $L$ is a sub-algebra of a Lie algebra
$L_0\oplus L$ where $L_0$ is a Lie algebra of graduation $0$, then
for any $x_0\in L_0$, the adjoint action of $x_0$ on $L$ defines a
graded derivation on $L$. Whenever this adjoint action is invertible,
we have a unique solution to the equation
\[
ad_{x_0}(\varphi)= [x_0,\varphi]=\varphi.x
\]
which is fundamental is the study of vector fields (see section \ref{secds}). But
conversely, the study of vector fields provides many ideas to deal
with the case of non invertible derivations.  Let us illustrate this on the toy model $L=y^2 \CC[[x]]\partial_y$ with $L_0=\{\alpha x\partial_x+\beta y \partial_y,\ \alpha,\beta\in \CC\}$. The associated derivation $d=ad_{X_{\alpha,\beta}}$ (with $X_{\alpha,\beta}=\alpha x\partial_x+\beta y \partial_y$) gives for $n\geqslant 0$
\[
\begin{array}{rcl}
d(x^ny^2\partial_y) &=& [\alpha x\partial_x+\beta y \partial_y,x^ny^2\partial_y]=(\alpha n+\beta)x^n y^2\partial_y \\
d(a(x)y^2 \partial_y) &=& (\alpha xa'(x)+\beta a(x))y^2\partial_y
\end{array}
\] and the derivation is invertible if $\alpha n+\beta$ does not vanish for $n\geqslant 0$ or, equivalently, when, for any $b(x)\in \CC[[x]]$, there exists a unique formal solution $a(x)$ to the equation
\[
 \alpha xa'(x)+\beta a(x)=b(x).
 \]
More precisely, for a fixed $X_{\alpha,\beta}=\alpha x\partial_x+\beta y \partial_y\in L_0$ and $y^2b(x)\partial_y$, if we look for an element $\exp(a(x)y^2 \partial_y)$ in the Lie group such that
\[
d(\exp(a(x)y^2 \partial_y))=\exp(a(x)y^2 \partial_y) . y^2 b(x)\partial_y
\]
then, since any Lie bracket in $y^2\CC[[x]]\partial_y$ vanishes, Theorem \ref{rightinv} ensures that $ \exp(a(x)y^2 \partial_y)$ is given by the solution of
\[
\alpha xa'(x)+\beta a(x)=b(x) 
\]
From the dynamical systems point of view, the reader can check that, when identifying $\exp(a(x)y^2 \partial_y)$ to the formal diffeomorphism 
\[
\psi(x,y)=(x,z)=\left(x,\frac{y}{1-a(x)y} \right),
\]
the (formal) solutions of the differential system
\[
\left\lbrace
\begin{array}{rcl}
\dsp \frac{dx}{dt} &=&\alpha x\\
\dsp \frac{dy}{dt} &=&\beta y
\end{array} \right.
\]
are in one--to--one correspondence with the solutions of 
  \[
\left\lbrace
\begin{array}{rcl}
\dsp \frac{dx}{dt} &=&\alpha x\\
\dsp \frac{dz}{dt} &=&\beta z +b(x)z^2
\end{array}\right. 
\]
since
\[
\begin{array}{rcl}
\dsp \frac{dz}{dt}  &=&\dsp \frac{d}{dt} \left(\frac{y}{1-a(x)y} \right) \\
 &=& \dsp \frac{dy}{dt} .\frac{1}{1-a(x)y} + \left(a'(x)y \frac{dx}{dt} +a(x)\frac{dy}{dt}\right).\frac{y}{(1-a(x)y)^2} \\
  &=& \dsp \frac{\beta y}{1-a(x)y} + (\alpha xa'(x)+\beta a(x))\frac{y^2}{(1-a(x)y)^2} \\
  &=& \beta z +b(x)z^2.
\end{array}
\]

Depending on the choice of $\alpha$ and $\beta$, the associated derivation could, in the best case, be invertible on the complete enveloping algebra and, in the worst case (for example $\beta=0$), be non-invertible on the Lie algebra. 


  
  
  


  
  


\section{The ``overall''  invertible case and Rota-Baxter algebras.}\label{secovin}


If $d$ is invertible as well on $L$ as on $\mathcal{U}_{\geq 1}$, we
will say that $d$ is ``overall'' invertible and it is natural to
transform the equation $d\varphi=\varphi x$ into the so-called Atkinson's recursion:
\begin{equation}
\varphi=1+I(\varphi x)
\end{equation}
where $I$ is the inverse of $d$ on  $\mathcal{U}_{\geq 1}$. In this
case,
\[\forall x,y \in \mathcal{U}_{\geq 1},\quad I(x)I(y)=I(I(x)y+xI(y)).
\]
Namely, $I$ is a Rota-Baxter operator of weight $0$:
\begin{de}
Let $A$ an associative algebra and $R$ a linear operator on $A$. $R$
is a Rota-Baxter operator of weight $\theta$ if:\[
\forall a,b \in A,\quad R(a)R(b)=R(R(a)b+bR(a))+\theta R(ab)
\]
\end{de}

In this case (see \cite{KMP}), the equation (with perturbative parameter
$\al$):
\begin{equation}
\varphi_{\al}=1+\al I(\varphi_{\al} x)
\end{equation}
admits a perturbative solution:
\begin{equation}
\varphi_{\al} = 1 + \sum_{n \geq 1} \al^n R^{[n]} (x)
\end{equation}
with $R^{[1]} (x) = I (x)$ and $R^{[n + 1]} (x) = I(R^{[n]} (x)
.x)$. This solution is in $G$ and its
logarithm, as well as its image by the Dynkin map $\log_Y$, are given by 
combinatorial formulas (see \cite{KMP}). We proved in \cite{PM}:
\begin{thm}\label{thstrigen}
  Let $\delta$ a graded derivation that commutes with $d$ (thus with
  $I$), the above perturbative solution satisfies:
\begin{equation}\label{strigen}
\log_{\delta} (\varphi_{\al}) = \varphi^{- 1}_{\al} . (\delta (\varphi_{\al})) = \sum_{n
     \geq 1} \al^n \mathcal{L}_{\delta}^{[n]} (x)
\end{equation}
  with $\mathcal{L}^{[1]}_{\delta} (x) = I (\delta (x))$ et $\mathcal{L}^{[n + 1]}_{\delta} (x) =
  I ([\mathcal{L}_{\delta}^{[n]} (x), x])$.
\end{thm}

The proof of this theorem is based on the fact that $d$ is overall invertible. This is necessary to
get the perturbative expansion of the solution $d\varphi=\varphi.x$
and then the formula \ref{strigen}. But this formula only involves the inverse of $d$ on $L$ and shall still hold if $d$ is only invertible on $L$. This idea could certainly lead to a new proof of theorem \ref{rightinv}, when $d$ is invertible on $L$ but not
necessary on $\mathcal{U}_{\geq 1}$. 

For example, in the case $L=L^k$, since the vector spaces (over $k$) $L_n$ and thus
$\mathcal{U}_n$ are finite dimensional, for a given parameter
$\eps$, the regularized operator $d+\eps Y$ is a derivation on the Lie algebra
$L^{k[[\varepsilon]][\eps^{-1}]}$ where $k[[\varepsilon]][\eps^{-1}]$ is the algebra of Laurent series in $\eps$ with coefficients in $k$.

On one hand, on each $\mathcal{U}^{k[[\varepsilon]][\eps^{-1}]}_n$ the restriction of
$d+\eps Y$ is invertible: $\mathcal{U}_n^{k[[\varepsilon]][\eps^{-1}]}$ is a finite--dimensional  $k[[\varepsilon]][\eps^{-1}]$--module and the determinant of the restriction of $d+\eps Y$ to this module is a non zero Laurent series. The derivation $d+\eps Y$ is thus invertible, of inverse $I_{\eps}$, on $\mathcal{U}_{\geq 1}^{k[[\varepsilon]][\eps^{-1}]}$. 

On the other hand, on each $L_n\subset L^{k[[\varepsilon]][\eps^{-1}]}_n$, since $d$
is invertible, $I_{\eps}$ still makes sense when $\eps=0$ and
coincides, for this value, with the inverse of $d$ (say
$I_0$).

The perturbative solution of $d\varphi=\varphi.x$ makes no sense since
$d$ is not necessarily overall invertible but, after the
regularization, $d+\eps Y$ is overall invertible and the perturbative
expansion of the solution of
$(d+\eps Y)\varphi_{\eps}=\varphi_{\eps}.x$ makes sense, especially
when $x$ is in $L^k$.  The formula \ref{strigen}, with $\delta$
commuting with $d+\eps Y$ still hold. If we take $\delta=Y$ then the
right part of this formula make sense when $\eps=0$ thus
$\log_Y\varphi_{\eps}$ and then  $\varphi_{\eps}$ can be evaluated
at $\eps=0$ and gives the solution in the invertible case! There is no contradiction here, this simply means that the expansion
giving the
perturbative solution may have poles, but once rewritten, for example in
a basis of $\mathcal{U}$, poles in $\eps$ cancel out. 

This "$\eps$--regularization" turn a derivation in an overall invertible regularized derivation. Such regularizations appear in quantum field theory and will provide a way to deal with non--invertible derivations, see section \ref{norenor}.

\section{The non-invertible case.}\label{secnoninv}

\subsection{$d$-conjugacy.}
Inspired by the situation in dynamical systems, whether $d$ is
invertible or not, it is relevant to study $d$-conjugate elements in $L$
\begin{prop}
Two elements in $\hL$ are $d$-conjugate if there exists $\varphi\in
\hGL$ such that\[
d(\varphi)+v.\varphi=\varphi.u
\]
we say that $\varphi$ d-conjugates $u$ to $v$. This is an
equivalence relation and we note $u\sim_d v$.
\end{prop}

If $\varphi$ $d$-conjugates $u$ to $v$, then $\varphi^{-1}$
conjugates $v$ to $u$ since:
\[
d(\varphi^{-1})=-\varphi^{-1}
d(\varphi)\varphi^{-1}=-\varphi^{-1}(\varphi u-v \varphi )
\varphi^{-1}=-u \varphi^{-1} +\varphi^{-1} v
\]
On the same way, if $\varphi$ $d$-conjugates $u$ to $v$ and $\psi$
$d$-conjugates $v$ to $w$, then $\psi\varphi$ $d$-conjugates $u$ to
$w$:
\begin{eqnarray*}
d(\psi \varphi) +w \psi \varphi  &=& d(\psi)\varphi +\psi d(\varphi) +w\psi \varphi  \\
 &=& (\psi v -w \psi)\varphi +\psi (\varphi u-v \varphi)+w\psi \varphi  \\
 &=& \psi \varphi u
\end{eqnarray*}

The idea of $d$-conjugacy comes from dynamical systems and is
motivated by the following fact: On one hand, if, for example $d$ is
invertible on $\hL$, then there is only one $d$-conjugacy class since
any $u$ in $\hL$ is conjugated to $0$ by $\varphi=\exp_d(u)$. On the other hand, whenever $d$
is not invertible, these equivalence classes are not trivial and their
classification or description is of great interest, especially in the framework of dynamical systems.

The complete classification, up to $d$-conjugacy, shall, in many
cases, be out of reach. For example, if $d=0$. the equation reads
\[
v \varphi=\varphi u.
\]
Using the graduation and $\varphi=\exp(\alpha)$ ($\alpha\in
\hL$), in the first graded components we get:
\[
v_1=u_1,\ v_2+ v_1 \alpha_1= u_2 + \alpha_1 u_1 
\]
thus  we must find $\alpha_1$ such that 
\[
[\alpha_1,v_1]=v_2 -u_2
\]
and the existence of a solution highly depends on the structure and the
relations in $L$. This last equation suggest, in order to prove that two elements $u$ and $v$ are $d$--conjugate, to look for a solution $\varphi=\exp(\alpha)$ of the equation:  
\begin{equation}\label{conjlog}
\log_d (\varphi)+\varphi^{-1} v \varphi=u.
\end{equation}
If $\alpha=\sum_{n \geq 1} \alpha_n$ and $\beta=\log_d(\varphi)=\sum_{n \geq 1} \beta_n$, let us recall that we already have (see proposition \ref{dexp} and equation \ref{adjgrad}):
\begin{equation}\label{loggrad}
\beta_n=d(\alpha_n)+\sum_{k=1}^{n-1}\sum_{i=1}^{n-k}\frac{(-1)^i}{(i+1)!}\sum_{n_1+..+n_i=n-k}ad_{\alpha_{n_1}}(
ad_{\alpha_{n_2}}...(ad_{\alpha_{n_i}}(d(\alpha_{k})))...) = d(\alpha_n)+P_{n-1}(\alpha,d(\alpha))
\end{equation}
where $P_{n-1}$ is a Lie polynomial in $\alpha_k$ and $d(\alpha_k)$ with $k<n$. On the same way, 
\begin{lem}\label{uexp} If $v$ and $\alpha$ are in $L$, then
\begin{equation} \label{adju}
\exp(-\alpha). v . \exp(\alpha)=\left ( \sum_{i\geq 0} \frac{(-1)^i}{i!}ad_\alpha^i \right )(v).
\end{equation}
\end{lem}
 The straightforward proof follows the same lines as for lemma \ref{adj} and if $\alpha=\sum_{n \geq 1} \alpha_n$, $v=\sum_{n \geq 1} v_n$ and $w=\exp(-\alpha). v . \exp(\alpha)=\sum_{n \geq 1} w_n$:
\begin{equation}\label{conjgrad}
w_n=v_n+\sum_{k=1}^{n-1}\sum_{i=1}^{n-k}\frac{(-1)^i}{i!}\sum_{n_1+..+n_i=n-k}ad_{\alpha_{n_1}}(
ad_{\alpha_{n_2}}...(ad_{\alpha_{n_i}}(v_{k}))...) = v_n+Q_{n-1}(\alpha,v)
\end{equation} 
where $Q_{n-1}$ is a Lie polynomial in $\alpha_k$ and $v_k$ with $k<n$. 

Thanks to equations \ref{loggrad} and \ref{conjgrad}, one can already give a particular element in each conjugacy class:

\begin{thm}\label{Fnor} Let $d$ a graded derivation on $L$ and $F=\prod_{n\geq 1}F_n$ a supplementary vector space of $d(L)$ in $L$. For any $u\in L$, there exists an element $u_F\in F$ which is a $d$-conjugate of $u$.
\end{thm}
\begin{proof} If $u=\sum_{n\geq 1} u_n$ we must find $u_F=v=\sum_{n\geq 1} v_n\in F$ and $\varphi=\exp(\alpha)\in G$ ($\alpha=\sum_{n\geq 1} \alpha_n$) such that
\[
\log_d(\exp(\alpha))+\exp(-\alpha)v \exp(\alpha)=u
\]
that is to say, for all $n>0$:
\begin{equation}\label{Fnorgrad}
d(\alpha_n)+P_{n-1}(\alpha,d(\alpha)) +v_n+Q_{n-1}(\alpha,v)=u_n.
\end{equation}
For $n=1$, this read $d(\alpha_1)=u_1 -v_1$. Let $p_F$ the projection on $F$, parallel to $d(L)$. This is a graded linear operator and if $v_1=p_F(u_1)\in F_1$ then $u_1-v_1$ is in $d(L_1)$ and one can find a solution to the equation $d(\alpha_1)=u_1 -v_1$. Suppose for a given $n>0$, that we have found $\alpha_1,\dots,\alpha_n$ in $L$ and $v_1,\dots, v_n$ in $F$ that solve equation \ref{Fnorgrad}. The polynomials $P_{n}(\alpha,d(\alpha))$ and $Q_{n}(\alpha,v)$ are well defined and, at order $n+1$, we must find $\alpha_{n+1}\in L_{n+1}$ and $v_{n+1}\in F_{n+1}$ such that
\[
d(\alpha_{n+1})=u_{n+1}-P_{n}(\alpha,d(\alpha))-Q_{n}(\alpha,v)-v_{n+1}
\]
As in the case $n=1$, we take $v_{n+1}=p_F(u_{n+1}-P_{n}(\alpha,d(\alpha))-Q_{n}(\alpha,v))$ and we can then find an element $\alpha_{n+1}$ that solves the equation. This ends the recursive construction of $\alpha$ and $u_F=v$.
\end{proof}   

Note that the element $u_F$ is not unique since there is a latitude in the determination of $\alpha$. There would be more to say if one could chose $F$ as a sub--Lie algebra, especially $\ker_L d$. This is the case if we assume that
\begin{equation}\label{dnormal}
L=\ker_L d \oplus d(\hL).
\end{equation} 
This hypothesis is natural dynamical systems and, in the sequel, we will deal with derivations satisfying equation (\ref{dnormal}). Under this assumption, the restriction of $d$ from $d(L)$ to itself is invertible and we note in the sequel $I$ this inverse, defined from $d(L)$ to itself.

From now on, we can choose $F=\ker_L d$ in the previous theorem and we note $p$ the projector on $\ker_L d$, parallel to $d(L)$. This choice allow to define and study "normal forms". 

Following the example of $L=y^2 \CC[[x]] \partial_y$ if $d=ad_{x\partial_x}$ then we have
\[
d(L)=xy^2\CC[[x]] \partial_y,\quad \ker_L d=\text{Vect}\{y^2 \partial_y\}\] so that $L=\ker_L d \oplus d(\hL)$. In this case, with $F=\ker_L d=L_1$, one can easily see, using equation \ref{Fnorgrad} and the very simple Lie bracket in $y^2 \CC[[x]] \partial_y$, that  $a(x)y^2\partial_y$ is $d$--conjugate to $a(0)y^2\partial_y$. As we shall see now, the choice $F=\ker_L d$ provides $d$--normal forms that play a crude role in dynamical systems.

\subsection{Normalization.} \label{secnor}

Let us look closer to the classification of $\hL$, up to
$d$-conjugacy, when $L=\ker_L d \oplus d(\hL)$.

\begin{thm}\label{thnor} For any $u\in\hL$, there exists $v\in \ker_L d$
  such that
\[ u\sim_d v .
\]
Moreover, $v,w \in \ker_L d$ are $d$-conjugated to $u$ if and only if
there exists $\varphi\in \exp(\ker_L d)$ such that
\[
v\varphi =\varphi w
\]
\end{thm}

Such elements are called $d$-normal forms of $u$.

\begin{proof} The existence of such normal forms follows from theorem \ref{Fnor} with $F=\ker_L d$. For the second
  part of this theorem, let us assume that $v$ is a $d$--normal form of $u$
  and let $\varphi \in \exp(\ker_L d)$. If   $v=\varphi^{-1} w \varphi $, that is to say $\varphi v=w \varphi $, then, thanks to lemma \ref{uexp}, $w$ is in the sub--Lie algebra $\ker_L d$ and, since $d(\varphi)=0$,
  \[ d(\varphi)+w \varphi=w \varphi= \varphi v\]
  thus $w$ is $d$-conjugated to $v$ and to $u$: $w$ is a $d$--normal form.

Conversely, if $v$ and $w$ are two normal forms, then they are 
$d$-conjugated: there exists $\varphi \in G$ such that
\[ 
\log_d(\varphi)+\varphi^{-1} w \varphi=v
\] and it remains to prove that $\varphi \in \exp(\ker_L d)$.

As in the proof of theorem \ref{Fnor}, if $\varphi=\exp
\al$, we have for $n\geq 1$: 
\[
d(\alpha_n)+P_{n-1}(\alpha,d(\alpha)) +w_n+Q_{n-1}(\alpha,w)=v_n.
\]

In graduation 1 this gives $d(\al_1)+w_1=v_1$ or $d(\al_1)=v_1-w_1 \in
\ker_L d \cap d(\hL)={0}$. Thus $w_1=v_1$ and $\al_1\in \ker_L d$. In
graduation 2 the equation
\[
d(\al_2)-\frac{1}{2}[\al_1, d(\al_1)] +w_2 -[\al_1, w_1]=v_2
\] reduces to
\[d(\al_2)=v_2-w_2 -[\al_1, w_1] \in \ker_L d \cap d(\hL)\] since $[\ker_L d,\ker_L d]\in \ker_L d$. Once again $d(\alpha_2)=0$ 
and $\al_2\in \ker_L d$. We get recursively that $\varphi \in \exp(\ker_L
d)$, using the fact that the polynomials $Q_{n-1}(\alpha,w)$ are Lie polynomials of elements of the sub-Lie algebra $\ker_L d$. 
\end{proof}

This theorem gives a description of the $d$-conjugacy classes : the
$d$-conjugacy classes are exactly the  classes of $\ker_L d$
for the classic conjugacy by the group $\exp(\ker_L d)$. As this will be used in section \ref{seccor}, this implies that if $u$ is $d$-conjugated to $0$, then $0$ is its unique normal form because 
$\varphi 0 \varphi^{-1}=0$ !

The second part of this theorem implies that there can be more than one normal form. Let us mention that, in the framework of dynamical systems, J. Ecalle
and B. Vallet (see \cite{EV1}, \cite{EV2}), provide a way to define for each $u\in L$ a unique normal form, with the help of a supplementary condition:

\begin{prop} Let $\del$ be an invertible derivation such that $\del(\ker_L
  d)=\ker_L d$, there exist a unique solution to the system:
\[
\left \{
\begin{array}{rcl}
d(\varphi)+  v \varphi &=&\varphi u \\
d(v) &=& 0 \\
p((\del(\varphi))\varphi^{-1}) &=& 0
\end{array} \right.
\]
\end{prop} 

Let us give an idea of the proof. Once a solution $(v,\varphi)$ of 
\[
\left \{
\begin{array}{rcl}
d(\varphi) &=& \varphi u - v\varphi \\
d(v) &=& 0 \\
\end{array} \right.
\] is given the other solutions are given by $(\psi v
\psi^{-1}, \psi \varphi)$ with $\psi\in \exp(\ker_L d)$. In order to
prove this proposition we have to show that the condition 
\[ p((\del( \psi\varphi))( \psi\varphi)^{-1}) = 0 
\]
determines a unique $\psi\in \exp(\ker_L d)$. But,
\[
(\del(\psi\varphi ))( \psi\varphi)^{-1}=\psi (\del \varphi) \varphi^{-1}\psi^{-1}+
 \del(\psi)\psi^{-1}=\psi (\del \varphi) \varphi^{-1}\psi^{-1} -\psi\del(\psi^{-1}) 
\]
Once again, using the graduation, if
$\del(\varphi)\varphi^{-1}=\beta\in L$ and $\psi^{-1}=\exp(\al)\in \exp(\ker_L d)$, we
get for $n=1$,
\[p(-\del(\al_1)+\beta_1)=0
\] thus $\del(\al_1)=p(\beta_1)$ that has a unique solution in
$\ker_L d$ since $\del$ is invertible and $\ker_L d$ is stable by
$\del$. The remainder of the proof would follow recursively, as in the previous theorems. For any such
 $\del$ we can define a map $N_{\delta}$ from $\hL$ to $\ker
 d$ (note that the conjugating map is also unique).
 
 The reader can notice that in the case of $d=ad_{x \partial_x}$ on $y^2\CC[[x]] \partial_y$ there is exactly one normal form of $a(x)y^2 \partial_y$, namely $a(0)y^2\partial_y$. This is mainly due to the extreme simplicity of this toy model but it does not reflect the complexity of the general case, for which the choice of a normal form and its computation are not so easy to deal with.

\section{Normalization and renormalization.} \label{norenor}

\subsection{Main theorem.}\label{secmt}

The aim of this section is to point out that similar procedures as those used in perturbative quantum field theory can be used to compute normal forms in the case of non-invertible derivations. In the Connes-Kreimer picture of renormalization (see \cite{ck1} and \cite{ck2}), the computation of Feynman integrals shall deliver a character $\varphi$ on the the graded Hopf algebra of 1PI Feynman graphs, that is an element of the Lie group associated to the complete graded Lie algebra of infinitesimal characters on the Hopf algebra of Feynman graphs $L^k_{FG}$ (with $k=\CC$). This character is defined by Feynman integrals that could be divergent but can be renormalized, using dimensional regularization and Birkhoff decomposition. 

More precisely, after dimensional regularization, the Feynman integrals define a character $\varphi(\varepsilon)$ with values in the algebra of Laurent series $\mathcal{A}=k[[\eps]][\eps^{-1}]$, that is to say an element of the Lie group $G^{\mathcal A}$ of the Lie algebra $L^{\mathcal A}$. The presence of poles at $0$ reflects the divergence of the initial character. The usual renormalization of such a group-like element, is then based on the Birkhoff decomposition: since $\mathcal{A}=\mathcal{A}^-\oplus
\mathcal{A}^+$ where $\mathcal{A}^-=\eps^{-1}k[\eps^{-1}]$ and
$\mathcal{A}^+=k[[\eps]]$ are two sub-algebras, it is possible to factorize $\varphi(\varepsilon)$, that is
\begin{equation}
\varphi(\varepsilon)=\varphi^-(\varepsilon)\varphi^+(\varepsilon)
\end{equation} where $\varphi^-(\varepsilon)$ (resp. $\varphi^+(\varepsilon)$) is in $G^{{\mathcal A}^-}$ (resp. $G^{{\mathcal A}^+}$) and the evaluation of $\varphi^+(\varepsilon)$ at $\varepsilon=0$ gives the renormalized character $\varphi^\text{ren}$. We refer to \cite{EFGM} for further details on the Birkhoff decomposition. 

Such ideas can be applied in order to find normal forms. As in perturbative quantum field theory, the solutions of 
\begin{equation}
d\varphi=\varphi u
\end{equation}
are ill-defined whenever the derivation $d$ is not invertible. One can then chose to "regularize" $d$ as follows. Let $d$ a derivation and $\del$ a derivation on $L=L^{k}$ such
that $\ker_L d$ is stable by $\del$ and the restriction of $\delta$ from $\ker_L d$ to itself is invertible. These operators $d$ and $\delta$ can be extended to
$L^{\mathcal A}$ by $\mathcal A$--linearity and we get a "$\delta$-renormalization" scheme that delivers a normal form:
\begin{thm}\label{noren} Let $d$ and $\delta$ two derivations on $L=L^k$ such that:
\begin{enumerate}
\item $L=\ker_L d \oplus d(L)$,
\item $\ker_L d$ is stable by $\del$,
\item The restriction of $\delta$ from $\ker_L d$ to $\ker_L d$ is invertible.
\end{enumerate}
For $u\in L=L^k\subset L^{\mathcal{A}}$ ($\mathcal{A}=k[[\eps]][\eps^{-1}]$), the equation 
\begin{equation}
(d+\varepsilon \delta)\varphi=\varphi u
\end{equation} has a solution $\varphi=\varphi(\varepsilon)\in G^{\mathcal{A}}$ and, after Birkhoff decomposition, 
$$
\varphi=\varphi^- \varphi^+\quad (\varphi^-, \varphi^+)\in G^{\mathcal{A}^-}\times G^{\mathcal{A}^+},
$$
the group-like element $\varphi^{\text{ren}}=\varphi^+(\varepsilon)\mid_{\varepsilon=0}$ conjugates $u$ to a normal form $\beta \in \ker_{L} d$. Moreover, we have $\del \varphi^-_{\eps}= \varphi^-_{\eps}(\eps^{-1} \beta) $.
\end{thm}

Note that this result still hold on a Lie algebra $L^{\mathcal{B}}$ by considering Laurent series in $\mathcal{A}=\mathcal{B}[[\eps]][\eps^{-1}]$. As we shall see in section \ref{locres}, the normal form $\beta$ is an analog to the $\beta$ function that appears in perturbative quantum field theory (see \cite{ck2}, \cite{EFGBP}).

In the case $d=ad_{x\partial_x}$ on $L=y^2\CC[[x]]\partial_y$, the reader can check that one can choose $\delta=ad_{x\partial_x+y\partial_y}$, that coincides with the grading $Y$. In this case, $d+\varepsilon \delta=ad_{(1+\varepsilon)x\partial_x+\varepsilon y\partial_y}$ and for a given element
\[
u=a(x)y^2\partial_y,\quad a(x)=\sum_{n\geq 0} a_n x^n 
\]
A solution of $(d+\varepsilon \delta)\varphi=\varphi u$ is given by $\varphi(\varepsilon)=\exp(b(x)y^2\partial_y)$ with
\[
b(x)=\sum_{n\geq 0} \frac{a_n}{n(1+\varepsilon)+\varepsilon} x^n = \frac{a_0}{\varepsilon}+\sum_{n\geq 1} \frac{a_n}{n(1+\varepsilon)+\varepsilon} x^n 
\]
where fractions in $\varepsilon$ identify with Laurent series and, thanks to the simplicity of the Lie algebra,
\[
\varphi^-=\exp\left(\frac{a_0}{\varepsilon}y^2 \partial_y\right ),\quad \varphi^+=\exp\left(\sum_{n\geq 1} \frac{a_n}{n(1+\varepsilon)+\varepsilon} x^n y^2 \partial_y\right )
\]
so that a normal form of $u=a(x)y^2\partial_y$ is
\[
\beta=(\varphi^-_{\eps})^{-1}(\varepsilon \del \varphi^-_{\eps})=\varepsilon \frac{a_0}{\varepsilon}y^2 \partial_y = a_0 y^2 \partial_y.
\]

\subsection{Proof of theorem \ref{noren}.}

Let $d$ a graded derivation on $L=L^{k}$ such that $L=\ker_L \oplus d(L)$ and $p=p_d$, $q=\id -p$ the respective projections on the kernel of $d$ on $L$ and its image. It is clear that $d$, $p$ and $q$ extend to $L^{\mathcal{A}}$ with $L^{\mathcal{A}}=\ker_{L^{\mathcal{A}}} d \oplus d(L^{\mathcal{A}})$. For example, if $u=\sum_{n\geq 1} u_n \in L^{\mathcal{A}}$,
\begin{equation} \label{LAcomp}
u_n=\sum_{m\geq M_n} \varepsilon^m u_{n,m},\quad M_n\in\mathbb{Z},\ u_{n,m}\in L^k_n
\end{equation} and $d(u)=\sum_{n\geq 1} d(u_n)$ with
\[
d(u_n)=\sum_{m\geq M_n} \varepsilon^m d(u_{n,m})
\]

The identity $L^{\mathcal{A}}=\ker_{L^{\mathcal{A}}} d \oplus d(L^{\mathcal{A}})$ implies that the restriction of $d$ from $d(L^k)$ (or $d(L^{\mathcal{A}})$) to itself is invertible and we note $I$ this graded inverse. On the same way $\delta$ can be extended to $L^{\mathcal{A}}$, its restriction from $\ker_L$ (or $\ker_{L^{\mathcal{A}}}$) to itself is invertible and we note abusively $\delta^{-1}$ its inverse on $\ker_L$ (or $\ker_{L^{\mathcal{A}}}$). Note that all these operators are graded.

The first step of the proof is based on the following lemma:
\begin{lem} The endomorphism $d+\varepsilon \delta$ is an invertible derivation on $L^{\mathcal{A}}$ and its (graded) inverse is the linear map $I_{\varepsilon}$ defined for $u\in L^{\mathcal{A}}$ by 
\begin{equation}
I_{\eps}(u)=\eps^{-1}\del^{-1}(p(u))+(\id -\del^{-1} \circ p\circ \del)\circ I \left (\sum_{k\geq 0} (-1)^k \eps^k
  (q\circ \del \circ I)^{\circ^k} (q(u))\right ).
\end{equation}
\end{lem}
\begin{proof} It is a matter of fact to check that $d+\varepsilon \delta$ and $I_{\varepsilon}$ are well-defined: as all the operators do respect the graduation, we can work on graded components $L^{\mathcal{A}}_n$ with elements as in equation (\ref{LAcomp}). It remains to prove that $(d+\varepsilon \delta)\circ I_{\varepsilon}= I_{\varepsilon}\circ(d+\varepsilon \delta)=\id_{L^{\mathcal{A}}}=\id$. Since  $\ker_{L^{\mathcal{A}}} d$ is stable by $\del$, $d\circ \delta^{-1}\circ p=0$ thus
\[
(d+\eps \del)(\eps^{-1}\del^{-1}(p(u)))=p(u)
\]
and,  when restricted to $d(L^{\mathcal{A}})$, 
\[
(d+\eps \del) \circ (\id -\del^{-1} \circ p\circ \del)\circ I =\id +\eps \del \circ I -\eps p\circ \del \circ I=\id +\eps q\circ \del \circ I.
\]
If we apply to
\[
\sum_{k\geq 0} (-1)^k \eps^k
  (q\circ \del \circ I)^{\circ^k} (q(u))\in d(L^{\mathcal{A}})
\]
it is now clear that
\[ (d+\eps
\del)(I_{\eps}(u))=p(u)+q(u)=u
\]
thus $(d+\eps \del)\circ I_\eps =\id$. In order to prove that $I_\eps \circ (d+\eps \del)=\id$, let us notice that\[
\eps^{-1} \del^{-1} \circ p\circ (d+\eps \del)=\del^{-1} \circ p\circ \delta,
\]
$I\circ q \circ d=q$ and $q\circ \del \circ  p=0$. We get

\begin{eqnarray*}
\sum_{k\geq 0} (-1)^k \eps^k
  (q\circ \del \circ I)^{\circ^k} \circ q\circ (d+\eps \del) & =& \sum_{k\geq 0} (-1)^k \eps^k
  (q\circ \del \circ I)^{\circ^k} \circ q\circ d \\
  & &- \sum_{k\geq 1} (-1)^k \eps^k
  (q\circ \del \circ I)^{\circ^{k-1}} \circ q\circ \del \\
  &=& q\circ d \\
  & &+\sum_{k\geq 1} (-1)^k \eps^k
  (q\circ \del \circ I)^{\circ^{k-1}} \circ (q\circ \del \circ I\circ q\circ d- q\circ \del) \\
  &=& q\circ d \\
  & &+\sum_{k\geq 1} (-1)^k \eps^k
  (q\circ \del \circ I)^{\circ^{k-1}} \circ (q\circ \del \circ  q- q\circ \del) \\
  &=& q\circ d \\
  & &-\sum_{k\geq 1} (-1)^k \eps^k
  (q\circ \del \circ I)^{\circ^{k-1}} \circ (q\circ \del \circ  p) \\
  &=& q\circ d 
  \end{eqnarray*}
and finally
\begin{eqnarray*}
I_\eps \circ (d+\eps \del) &=& \del^{-1} \circ p\circ \delta +(\id -\del^{-1} \circ p\circ \del)\circ I \circ q \circ d \\
&=& \del^{-1} \circ p\circ \del+(\id -\del^{-1} \circ p\circ \del)\circ q \\
&=& \del^{-1} \circ p\circ \del \circ p +q = p+q =\id \\
\end{eqnarray*}
\end{proof}

Note that the expression of this inverse can be simplified if we assume that $d$ and $\delta$ commute.  Since $d+\eps \delta$ is invertible on $L^{\mathcal{A}}$, thanks to theorem \ref{rightinv}, the equation $(d+\eps \delta)\varphi =\varphi u$ has a solution $\varphi=\varphi(\eps)\in G^{\mathcal{A}}$ for any $u$ in $L^{\mathcal{A}}$ or $L$. Thanks to the Birkhoff decomposition (see \cite{EFGM}), we have
\[ \varphi_(\eps)=\varphi^-(\eps) \varphi^+(\eps) ,\quad
\varphi^{\pm}(\eps)=\exp \left(\al^{\pm}(\eps) \right),\ \al^{\pm}(\eps) \in
L^{\mathcal{A}^\pm}\]
If $u\in L$, since $(d+\eps \delta)\varphi =\varphi u$ we get
\begin{eqnarray*}
(d+\eps \delta)\varphi &=& (d+\eps \delta)(\varphi^- \varphi^+) \\
 &=& \varphi^-\left((d+\eps \delta)\varphi^+\right)+\left((d+\eps \delta)\varphi^-\right)\varphi^+ \\
 &=& \varphi^- \varphi^+ u
\end{eqnarray*}
and then
\begin{equation} \label{eqbir}
\left((d+\eps \delta)\varphi^+\right)(\varphi^+)^{-1}+(\varphi^-)^{-1} \left((d+\eps \delta)\varphi^-\right)=\varphi^+ u(\varphi^+)^{-1}
\end{equation}

But $(\varphi^-)^{-1}((d+\eps
\del)\varphi^-)$ is in
$L^{k[\eps^{-1}]}$ whereas the other terms are in  $L^{k[[\eps]]}$
thus these two parts of the identity do not depend on $\eps$. Let $\beta=(\varphi^-)^{-1}\left((d+\eps
\del)\varphi^-\right) \in L^k$ (not in $L^{\mathcal{A}}$!), then
\[((d+\eps \del)\varphi^+ +\beta \varphi^+=\varphi^+ u\]
and for $\eps=0$ the element $\varphi^\text{ren}=\varphi^+(0)$ conjugates $u$  to $\beta$. It remains to prove that $\beta$ is a normal form, that is $d(\beta)=0$. 
Note that if $\varphi^-=\exp(\alpha)$ with $d(\alpha)=0$, then $\del \varphi^-= \varphi^-(\eps^{-1} \beta) $ with $d(\beta)=0$ (see lemma \ref{dexp}). The proof of this last property is based on the following lemma: 
\begin{lem} \label{lemconst} Let $\psi\in G^{\mathcal{A}^-}$. If \[
\log_{d+\eps \del}(\psi)=\psi^{-1}. (d+\eps \del)(\psi)\in L^k
\]
then $\psi\in \exp(\ker_{L^{\mathcal{A}^-}} d)$
\end{lem}
\begin{proof} The proof of the lemma is based on the following observation: If, for a given $n\geq 1$, 
\[x_n=\sum_{M\leq k\leq -1} \eps^k x_{n,k}\in L^{\mathcal{A}^-}_n
\]
is such that $(d+\eps \del)(x_n)=y_n+z_n \in L^{\mathcal{A}}_n$ with $y_n\in d(L^k_n)$ and \[
z_n =\sum_{M\leq k\leq 0} \eps^k z_{n,k}\in  \ker_{L^{k[\eps^{-1}]}} d 
\]
then $y_n=0$ and $x_n=I_{\eps}(z_n)\in \ker_{L^{\mathcal{A}^-}} d$.
One can check that $z_{n,k}\in \ker_{L^k} d$ and
\[
I_{\eps}(z_n) =\eps^{-1} \del^{-1}(z_n)=\sum_{M-1\leq k\leq -1} \eps^k \del^{-1}(z_{n,k+1}) \in \ker_{L^{\mathcal{A}^-}} d\cap L^{\mathcal{A}^-}_n
\]
On the other hand, if $y_n\in d(L^k_n)$,
\[I_{\eps}(y_n) =(\id -\del^{-1} \circ p\circ \del)\circ I \left (\sum_{k\geq 0} (-1)^k \eps^k
  (q\circ \del \circ I)^{\circ^k} (y_n)\right )=x-I_{\eps}(z_n)\in L^{\mathcal{A}^+}_n \cap L^{\mathcal{A}^-}_n.
\]
It follows that necessarily $I_{\eps}(y_n)=0$ (thus $y_n=0$) and $x_n=I_{\eps}(z_n)\in \ker_{L^{\mathcal{A}^-}} d$.

This is the key to prove the lemma. If $\beta=(\psi)^{-1} \left((d+\eps
\del)\psi \right)=\log_{d+\eps \del} (\psi) \in L^k$ with
$\psi=\exp(\al)$ ($\al\in L^{\mathcal{A}^-}$) then, using lemma \ref{dexp} and the graduation (see equation (\ref{loggrad})), we get:
\[
\beta_n = (d+\eps \del)(\alpha_n)+P_{n-1}(\alpha,(d+\eps \del)(\alpha)) \in L^k_n
\]
where $P_{n-1}$ is a Lie polynomial in $\alpha_k$ and $(d+\eps \del)(\alpha_k)$ with $k<n$.

In graduation 1:
\[(d+\eps \del)\al_1=\beta_1\] and, thanks to the previous result (with $y_1=q(\beta_1) \in d(L^k_1)$ and $z_1=p(\beta_1)\in \ker_{L^{k[\eps^{-1}]}} d$), $q(\beta_1)=0$ and finally $\beta_1 \in \ker_{L^k} d$,  $\al_1 \in \ker_{L^{\mathcal{A}^-}} d$ and $\del \al_1=\eps^{-1} \beta_1$.
 Let us suppose now that
 \[
 d(\al_1)=\dots=d(\alpha_{n-1})=0=d(\beta_1)=\dots=d(\beta_{n-1})
 \]
 we have
 \[(d+\eps \del)(\alpha_n)=\beta_n-P_{n-1}(\alpha,(d+\eps \del)(\alpha))=\beta_n-P_{n-1}(\alpha,\eps \del(\alpha))
 \]
 with
 \[ P_{n-1}(\alpha,\eps \del(\alpha))=\sum_{k=1}^{n-1}\sum_{i=1}^{n-k}\frac{(-1)^i}{(i+1)!}\sum_{n_1+..+n_i=n-k}ad_{\alpha_{n_1}}(
ad_{\alpha_{n_2}}...(ad_{\alpha_{n_i}}(\varepsilon \delta(\alpha_{k})))...) 
\]
Since, for any $k$, $\alpha_k\in L^{\mathcal{A}^-}$, it is clear that we can set
\[
\begin{array}{rcl}
x_n &=& \alpha_n \in L^{\mathcal{A}^-}_n \\
z_n &=& p(\beta_n)-P_{n-1}(\alpha,\eps \del(\alpha)) \in \ker_{L^{k[\eps^{-1}]}} d\\
y_n &=& q(\beta_n) \in  d(L^k_n) \\
\end{array}
\]
and we already proved that $y_n=0$ (thus $d(\beta_n)=0$) and $x_n=\alpha_n \in \ker_{L^{\mathcal{A}^-}} d$

It follow recursively that $d(\alpha_k)=0$, $\beta_k\in \ker d$ for $k\geq 1$
\end{proof}

Using this lemma, together with the Birkhoff decomposition, $\beta=\log_{d+\eps \del} \varphi^- \in L^k$ is a normal form related to  $\varphi^-$ by 
\begin{equation} \label{resid}
\del \varphi^-= \varphi^-(\eps^{-1} \beta) 
\end{equation}

\section{Further developments.}\label{secmore}

We enclose in this section many ideas, developments that are suggested by the strong similarity between our results, stemming from dynamical systems, and the algebraic machinery developed in the context of perturbative quantum field theory.
 
\subsection{Locality and Residues} \label{locres}
 The last identity (\ref{resid}) has to be related to residues. If we write
 \[
 \varphi^-=1+\sum_{k\geq 1}\eps^{-k} \varphi^-_k, \quad \varphi^-_k \in \mathcal{U}^{k}
 \]
 then, the residue $\Res \varphi^-=\varphi^-_1$ is such that
 \[
 \del \Res \varphi^-=\beta
 \]
and it looks very close to the beta function in perturbative quantum field theory, see \cite{ck2}, \cite{EFGBP} and \cite{Man}. This concept is related to the "locality" of counter terms (that is $\varphi^-$). In our framework one can define on $L^{\mathcal{A}[[\tau]]}$ (and then on $\mathcal{U}^{\mathcal{A}[[\tau]]}$) the automorphism $\theta_\tau$:
\[\theta_\tau (x)=e^{\tau \eps \del} x
\]
Following section 7 in \cite{EFGBP} (in the case $\del=Y$), $\varphi \in G^{\mathcal{A}}$ is $\del$--local if for any $\tau$, the Birkhoff decomposition of $\varphi^\tau=\theta_\tau (\varphi)$ is such that 
\[
\partial_\tau (\varphi^\tau)^-=0
\]   
that is the counter terms do not depend on $\tau$. Under the assumption that $[d,\delta]=0$ (thus $[d,\theta_\tau ]=0$), the same results as in \cite{EFGBP} could  be obtained here and one can check that our renormalization procedure delivers a $\del$--local element. In fact, we have $\partial_\tau \varphi^\tau=\theta_\tau (\eps \del \varphi)$ and
\[
(d+\partial_\tau)\varphi^\tau=\theta_\tau \left((d+\eps \del)(\varphi)\right)=\varphi^\tau u^\tau
\]
and after Birkhoff decomposition (as in equation (\ref{eqbir})):
\[
\left((d+\partial_\tau)(\varphi^\tau)^+\right)((\varphi^\tau)^+)^{-1}+((\varphi^\tau)^-)^{-1} \left((d+\partial_\tau)(\varphi^\tau)^-\right)=(\varphi^\tau)^+ u^\tau((\varphi^\tau)^+)^{-1}
\]
but as in the proof of theorem \ref{noren}, this identity shows that  $((\varphi^\tau)^-)^{-1} \left((d+\partial_\tau)(\varphi^\tau)^-\right)\in L^{\mathcal{A}[[\tau]]^-}\cap  L^{\mathcal{A}[[\tau]]^+}=\lbrace 0 \rbrace$ and $(d+\partial_\tau)(\varphi^\tau)^-=0$. If we expand in powers of $\tau$, since $(\varphi^0)^-=\varphi^-$ with $d(\varphi^-)=0$, we get
\[
\partial_\tau (\varphi^\tau)^-=0
\]
and this proves the $\del$--locality of $\varphi$.

As we can see here, this renormalization is very similar to the one used in physics. Note also that it could be of some interest to consider "perturbations" of $d$ such as $d+\eps \del_1 +\eps^2 \del _2 +\dots$. On the same way, the following section suggests that other renormalization procedure appear in the framework of dynamical systems.

\subsection{Another "renormalization" : the correction.} \label{seccor}

As pointed out in \cite{EFP}, in pQFT, the attempted (but ill-defined) group-like element $\varphi$ is associated to a Lagrangian $\mathcal{L}$ and the renormalization procedure can be interpreted as an iterative process, based on the graduation of the considered Hopf algebra, that consists in modifying the Lagrangian:
\[
\mathcal{L} \rightarrow \mathcal{L}-\mathcal{L}_1 \rightarrow \mathcal{L}-\mathcal{L}_1 -\mathcal{L}_2 \rightarrow \dots \rightarrow \mathcal{L}^{ren}
\]
so that the renormalized group-like element $\varphi^{ren}$ corresponds to the modified Lagrangian $\mathcal{L}^{ren}$.

The same principle appear in the framework of dynamical systems but the final object differs from a normal form, this is the correction:
\begin{thm} Let $d$ a derivation such that $L=\ker d \oplus d(L)$. For any $u\in \hL$ there exists a unique $u^c \in 
  \ker d$
  such that $u-u^c$ is in the $d$-conjugacy class of $0$. $u_c$ is
  called the correction of $u$. Moreover, if $u \in \ker d$, then $u^c=u$.
\end{thm}

As in the Lagrangian interpretation of renormalization in pQFT, the idea of the proof is to modify $u$, graded component by graded component, so that we can solve the equation:
\[
d\varphi=\varphi (u-u^c_1-u^c_2 -\dots)
\]

\begin{proof}
 Whenever there exists $\varphi$ such that $\log_d(\varphi)=u$, then we
can take $u^c=0$, but, when $d$ is not invertible, such $\varphi$ does
not necessarily exists. Let us consider
\[
u=\sum u_n,\  u^c=\sum u^c_n, \ u-u^c=v=\sum v_n, \
\varphi=\exp\left ( \sum \alpha_n \right ) 
\]
The equation $\log_d(\varphi)=v=u-u^c$ reads (see lemma \ref{dexp} and equation (\ref{loggrad})):
\[
d(\al_n)=v_n - P_{n-1}(\al,d(\al))
\]
where \[
P_{n-1}(\al,d(\al)=\sum_{k=1}^{n-1}\sum_{i=1}^{n-k}\frac{(-1)^i}{i+1}\sum_{n_1+..+n_i=n-k}ad_{\al_{n_1}}(
ad_{\al_{n_2}}...(ad_{\al_{n_i}}(d(\al_{k})))...)
\]
depends only on $\al_1,...,\al_{n-1},d(\al_1),...,d(\al_{n-1})$ (with $P_0=0$).
For $n=1$, we must solve:
\[
d(\al_1)=v_1=u_1 -u^c_1
 \]
thus $u^c_1=p(u_1)\in \ker d$ (unique), $v_1=q(u_1)$ and then we can chose $\al_1$ up to an element
of $\ker d$ : $\al_1=I(q(u_1))+z_1$, $z_1\in \ker d$.
Now suppose that  for $1 \leq k \leq n$ we can find
$u^c_1,...,u^c_n$ so that there exists solutions (say $\al_1,
...,\al_n$) to the equations:
\[
\forall  1 \leq k \leq n,\quad d(\al_k)=u_k -u^c_k -
P_{k-1}(\al,d(\al))
\]
for the next equation:
\[d(\al_{n+1})=u_{n+1}-u^c_{n+1}-P_n(\al,d(\al)), 
\]
as $L=\ker d \oplus d(L)$, we must define
\[
u^c_{n+1}=p \left ( u_{n+1}-P_n(\al,d(\al)) \right )
\]
and this ends the (recursive) proof of the existence of $u^c$. But in 
this recursive definition $u^c$ seems to depend on the choice of
$\al$. Indeed $u^c$ is unique : suppose we have $u^1$ and $u^2$ in
$\ker d$ such that $u-u^1\sim_d 0 \sim_d u-u^2$, then there exists
$\varphi  $, $\psi  $ such that:
\[
d(\varphi)=\varphi(u-u^1)  \ , \ d(\psi)=\psi(u-u^2) 
\]
In $\hGL$, let $ \varphi \psi^{-1}=\phi$, then
\[ 
d(\varphi)=\varphi(u-u^1) =\phi \psi(u-u^1)  = d(\phi) \psi+\phi
d(\psi)=
d(\phi)\psi +\phi \psi (u-u^2)
\]
thus
\[
\phi \psi (u^2 -u^1)=d(\phi)\psi
\] or rather
\[
\psi (u^2 -u^1) \psi^{-1} = \log_d(\phi)
\]
Using lemmas \ref{dexp} and \ref{uexp} with $\phi=\exp \al$ and
$\psi^{-1}=\exp \beta$:
\[
\sum_{i\geq 0}\frac{1}{i!}  ad_\beta^i(u^2-u^1)=\sum_{i\geq 0} \frac{(-1)^i}{(i+1)!}ad_\al^i(d(\al)).
\]
Thanks to the graduation we get first $d(\al_1)=u_1^2-u_1^1$: 
since it is in the kernel and the
image of $d$, we get $d(\al_1) =0$ and $u_1^1-u_1^2=0$. One can then
prove recursively on $n$ that, for $1\leq k \leq n$:
\[
d(\al_k)=0\ ,\ u_k^1=u_k^2.
\]
If this is true then the equation reduces, in graduation $n+1$ to
\[
d(\al_{n+1})=u_{n+1}^2 -u_{n+1}^1
\]
thus $d(\al_{n+1})=u_{n+1}^2 -u_{n+1}^1=0$ and this ends the
recursion.
\end{proof}

Note that, $\varphi$ is unique if we assume that $\varphi=\exp(\al)$
with $\al\in d(\hL)$, using the inverse $I$ of the restriction of $d$ from $d(L)$ to $d(L)$. 

We don't know if this correction can be interpreted in the framework pQFT but let us just end by two remarks on the link between normal forms obtained by the "renormalization scheme" described in theorem \ref{noren} and the correction. For a given $\del$ as in theorem \ref{noren}, we get a function $N_{\del}$ from $L$ to $\ker_L d$ such that, for all $u\in L$, $u\sim_d N_{\del}(u)$ and the correction associates to any $u\in L$ a unique $u^c=C(u)$, independent of $\del$, such that  $u-C(u)\sim_d 0$. As pointed out in section  \ref{secnor}, there is no other normal form than $0$ which  is conjugated to $0$:
\[
N_{\del}(u-C(u))=0
\]
thus $u-C(u)$ is a fixed point for the operator $\id_L-N_{\del}$. One could even go further: if we introduce a parameter $t$
then,
\[ \lim_{t \rightarrow 0} t^{-1} C(tu)=   \lim_{t \rightarrow 0}
t^{-1} N_{\del}(tu)=p(u)
\]
so that, in a first approximation, the normal form and the correction coincide. This looks like two renormalizations, with different prescriptions
but with a common ``minimal substraction'' $u\mapsto u-p(u)$.

Under the same hypothesis as in theorem \ref{noren}, since $d+\eps \del$ is invertible on $L^{\mathcal{A}}$ with $\mathcal{A}=k[[\eps]][\eps^{-1}]$, the equation $\log_{d+\eps\del}\varphi=u\in L^k$ has a unique solution $\varphi=\exp(\alpha)\in G^{\mathcal{A}}$. As for the Birkhoff decomposition, the reader can check that there exist $(\al^+,\al^-)\in L^{\mathcal{A}^+}\times L^{\mathcal{A}^-}$ such that, if $\varphi^{\pm}=\exp(\al^{\pm})$,
\begin{equation}
 \log_{d+\eps\del}\varphi=\log_{d+\eps\del}\varphi^+ +\log_{d+\eps\del}\varphi^-=u.
 \end{equation}
 As before, the recursive proof is based on equation (\ref{loggrad}): For $n\geq 1$,
 \begin{eqnarray*}
(d+\eps \del)(\al_n) +P_{n-1}(\al,(d+\eps \del)(\al))&=&(d+\eps \del)(\al^+_n) +P_{n-1}(\al^+,(d+\eps \del)(\al^+))\\
 & & +(d+\eps \del)(\al^-_n)+P_{n-1}(\al^-,(d+\eps \del)(\al^-))
\end{eqnarray*}
thus
\[
\al^+_n+\al^-_n=\al_n+I_{\eps} \left( P_{n-1}(\al,(d+\eps \del)(\al))-P_{n-1}(\al^+,(d+\eps \del)(\al^+))-P_{n-1}(\al^-,(d+\eps \del)(\al^-))\right).
\]
But $\log_{d+\eps\del}\varphi^-\in L^{k[\eps^{-1}]}$, $\log_{d+\eps\del}\varphi^+\in L^{k[[\eps]]}$ and their sum is $u\in L^k$. This means that $\gamma=\log_{d+\eps\del}\varphi^-\in L^k$ and, thanks to lemma \ref{lemconst}, $d(\gamma)=0$ thus
\[
\log_{d+\eps\del}\varphi^+=u-\gamma \text{ with } d(\gamma)=0
\]and this identity makes sense at $\eps=0$. We recover the correction $\gamma=u_c$ by this "additive" Birkhoff decomposition.

\section{Conclusion.}\label{secconc}

This use of "renormalization schemes" is new in the framework of dynamical systems and shall open new perspectives in this area. It shall interacts with J. Ecalle's mould calculus that provides a way to "lift" the computation of diffeomorphisms to computations in the Lie group of the Lie algebra of infinitesimal characters on a shuffle Hopf algebra (see \cite{SNAG}, \cite{FM}). If mould calculus is sufficient to lead to explicit computations, this does not tackle the difficult and crucial question of the analyticity of the computed diffeomorphisms.

Surprisingly, the analytic regularity of such diffeomorphism can sometimes be easily obtained by computing in a Hopf algebras of trees, which is strongly related to the Prelie structure of vector fields (see \cite{FM}). There shall be no major difficulty to adapt the results of this paper to the case of derivations on Prelie algebras.

This may even be the right theoretical framework to deal with the similar questions in "discrete dynamical systems" that are related to  group equations 
\[
\theta \varphi=\psi\varphi, \quad \psi,\varphi\in G
\]
where $\theta$ is a graded automorphism on $L$: $\theta
([x,y])=[\theta(x),\theta(y)]$, with an "invertibility" condition on $\theta-\id_L$. This is the right framework to deal with Hurwitz multizetas (see \cite{OB}) and their regularization. 

Let us end this paper with a remark on the choice of an $\eps$-regularization. In section \ref{secmt}, we illustrated theorem \ref{noren} in the case $d=ad_{x\partial_x}$ on $L=y^2\CC[[x]]\partial_y$ with $\delta=ad_{x\partial_x+y\partial_y}$ and the solution of 
\[
\log_{d+\eps \del} \varphi=a(x)y^2 
\]
was $\varphi(\varepsilon)=\exp(b(x)y^2\partial_y)$ where
\[
(1+\eps)x b'(x)+ \eps b(x)=a(x).
\] 
Theorem \ref{noren} gave then the normal form $a(0)y^2\partial_y$. Following \cite{MenDE}, we could have chosen the regularization $d_\eps=ad_{x^{1-\eps}\partial_x}$, assuming that one can now work on $\CC[[x,x^{\eps}]]$. After expansion $x^\eps=\sum_{n\geq 0}\frac{\eps^n \log^n x}{n!}$ and Birkhoff decomposition, it appear that there exists solutions to the equation
\[
\log_{d} \psi=a(x)y^2 , \quad \psi=\exp(c(x)y^2\partial_y) 
\]
but with a series $c(x)$ in $x$ \textbf{and} $\log x$. On one hand, this result is very close to the results of J. Ecalle on "ramified linearization" (see \cite{ERL}). On the other hand, the apparition of logarithmic terms shall be familiar to the experts of pQFT and shall be related to the toy model described in \cite{EFP} (section 4.2).  

The research leading these results was partially supported by the French National Research Agency under the reference ANR-12-BS01-0017.

\section{Appendix: Dynamical systems} \label{secds}
Let us first recall some results on autonomous analytic local differential
equations (see {\cite{Ilya}} for proofs and details), that is a system (in
dimension $\nu$)
\begin{equation}
  \dot{x} = f (x)
\end{equation}
where $x = (x_1, \ldots, x_{\nu})$ and $f = (f_1, \ldots, f_{\nu})$ is an
analytic map defined in a neighborhood of the origin of $\mathbbm{C}^{\nu}$,
with values in $\mathbbm{C}^{\nu}$ such that $f (0) = 0$. There are several
strategies to study such a local dynamical system. The first one is to compute
the flow (whose existence, uniqueness and analyticity is due to the
Cauchy-Lipschitz theorem), that is the local biholomorphism from
$\mathbbm{C}^{1 + \nu}$ to \ $\mathbbm{C}^{1 + \nu}$ :
\[ \varphi (t, x) = (t, \varphi^t (x)) \]
that describes the local solution, of the system $(t, y (t))$ such that $y (0)
= x$. Note that, since the system is autonomous ($f$ does not depend on $t$),
for small values $t$ and $s$, we have $\varphi^t \circ \varphi^s = \varphi^{t
+ s}$. Although this flow encodes all the properties of the solutions, its
computation does not give many information on the geometry of the solutions.
Since Poincar\'e, in order to get more geometrical information, a better
approach is to study this solutions up to a change of coordinates. Let $\psi$
be a local ($\psi (0) = 0$) biholomorphism of $\mathbbm{C}^{\nu}$, then if $y
(t) = \psi (x (t))$ where $x$ is a solution of the initial system, then :
\[ \dot{y} = \sum_{i = 1}^{\nu} \dot{x}_i \frac{\partial \psi}{\partial x_i}
   (x) = \sum_{i = 1}^{\nu} f_i (x) \frac{\partial \psi}{\partial x_i} (x) = g
   (y) = g (\psi (x)) \]
and the vector field $f$ is transformed into the vector field $g$ by the
equation
\begin{equation}
  \sum_{i = 1}^{\nu} f_i (x) \frac{\partial \psi}{\partial x_i} (x) = g (\psi
  (x))
\end{equation}
This defines an equivalence relation of vector field up to local
biholomorphism and, in order to grab the geometry of the solutions, it seem
reasonable to try to find the most simple vector field $g$ which is
conjugated to $f$. For the sake of simplicity we will suppose in the sequel
that the {\tmem{linear part}} of $f$ is diagonal:
\[ f (x) = \Lambda x + u (x) \]
where $\Lambda$ is a diagonal matrix whose diagonal $\lambda = (\lambda_1,
\ldots, \lambda_{\nu})$ is called the {\tmem{spectrum}}, and the components
of $u = (u_{1,} \ldots, u_{\nu})$ are of valuation at least 2 (for the usual
valuation in $\mathbbm{C}\{x\}$). In this context (as in the context of
perturbative quantum field theory), it seem reasonable to consider $f$ as a
{\tmem{perturbation}} of its linear part and to ask if it is conjugated to its
linear part: Does there exist a local biholomorphism of $\mathbbm{C}^{\nu}$,
tangent to the identity (in order to preserve the linear part) such that:
\[ \sum_{i = 1}^{\nu} f_i (x) \frac{\partial \psi}{\partial x_i} (x) = \sum_{i
   = 1}^{\nu} \lambda_i x_i \frac{\partial \psi}{\partial x_i} (x) + \sum_{i =
   1}^{\nu} u_i (x) \frac{\partial \psi}{\partial x_i} (x) = \Lambda . \psi
   (x) \]
This equation is called the homological equation and it provides recursive
equations on the coefficients of $\psi$. Unfortunately (see for example
{\cite{Mart}} or {\cite{Ilya}}) many difficulties arise in the resolution of
such equation (see also below):
\begin{enumeratenumeric}
  \item If, for any $1 \leqslant j \leqslant \nu$ and any sequence of
  non-negative integers $n = (n_1, \ldots, n_{\nu})$ with $|n| = n_1 + \ldots
  + n_{\nu} \geqslant 2$, we have:
  \begin{equation}
    \left\langle \lambda, n \right\rangle - \lambda_j \not= 0 \hspace{1em} (
    \left\langle \lambda, n \right\rangle = \lambda_1 n_1 + \ldots +
    \lambda_{\nu} n_{\nu}) \label{nonres}
  \end{equation}
  then the vector field is {\tmem{non-resonant}} and one can solve recursively
  the homological equation. Unfortunately, the series $\psi$ is not
  necessarily convergent (i.e. analytic), unless the spectrum satisfies some
  {\tmem{diophantine condition}} (see {\cite{Mart}}).
  
  \item Otherwise, when there are some cancellations of the values \
  $\left\langle \lambda, n \right\rangle - \lambda_j$, then the homological
  equation cannot {\tmem{even formally}} be solved. In this case, the case of
  resonant vector fields, solving recursively the homological equation
  involves some divisions by $0$. This also means that, in this case, the
  vector field cannot be {\tmem{linearized}} : at best, one can try to
  conjugate $f$ to a simpler vector field (but not its linear part) which is
  called a normal (or pre-normal) vector field $g (x) = \Lambda .x + v (x)$
  such that the components $v_j$ of $v$ contains only monomials $x^n =
  x_1^{n_1} \ldots x_{\nu}^{n_{\nu}}$ such that $\left\langle \lambda, n
  \right\rangle - \lambda_j = 0$. This is indeed possible but nothing ensures
  once again that we get an analytic solution, either for $\psi$ nor for $g$.
\end{enumeratenumeric}
This small survey of the situation calls for some remarks. We shall leave
aside the analytic difficulties (for which their still exists open problems)
and focus on the algebraic difficulties, that is computing formal series
rather than analytic series. Even formally, there are some obstruction to the
resolution of the homological equation: we started with the naive guess that,
since $f$ can be seen as perturbation of its linear part, it shall be
conjugated to it. This look like the situation in perturbative quantum field
theory (pQFT), where we see the action as a perturbation of the free action,
and as in pQFT, this gives a ill-defined solution (here the map $\psi$). The
analogy between resonant vector field and pQFT goes even further: in the resonant case, the computation of a normal form
can be obtained by a renormalization procedure (as in pQFT).

As pointed out in {\cite{ck1}} and {\cite{ck2}}, pQFT leads to the computation
of a character (maybe ill-defined) on a graded commutative Hopf algebra, that is to
say to the computation of an element of the Lie group of a complete graded Lie algebra.
The situation is indeed the same in the study of {\tmem{formal}} differential
equation, when considering vector fields as {\tmem{operators}}: let $f$ is a
formal vector field without constant term ($f \in \mathbbm{C}_{\geqslant 1}
[[x_1, \ldots, x_{\nu}]]^{\nu} =\mathbbm{C}_{\geqslant 1} [[x]]^{\nu}$), and
$A \in \mathbbm{C}[[x]]$ a formal series, if $\dot{x} = f (x)$ and $y = A (x)$
then:
\[ \dot{y} = \dot{A (x)} = \sum_{i = 1}^{\nu} \dot{x}_i \frac{\partial
   A}{\partial x_i} (x) = \left( \sum_{i = 1}^{\nu} f_i (x)
   \frac{\partial}{\partial x_i} \right) .A (x) \]
so that we can identify, in a unique way, a vector field $f$ with a
differential operator (even a derivation) $X = X_f$ {\tmem{acting}} on
$\mathbbm{C}[[x]]$:
\[ \sum_{i = 1}^{\nu} f_i (x) \frac{\partial}{\partial x_i} = X (x) \]
On the same way, if $\psi$ is a local formal, invertible diffeomorphism of
$\mathbbm{C}^{\nu}$, that is to say 
\[\psi \in 
\mathbbm{C}_{\geqslant 1} [[x_1, \ldots, x_{\nu}]]^{\nu}
=\mathbbm{C}_{\geqslant 1} [[x]]^{\nu}
\]
with an invertible linear part
$\left( \left( \frac{\partial \psi_i}{\partial x_j} (0) \right) \right)$, then
one can associate an invertible operator (called a substitution automorphism)
$F_{\psi}$ acting on $\mathbbm{C}[[x]]$:
\[ (F_{\psi} .A) (x) = A (\psi (x)) \]
Moreover, it is a matter of fact to check that a linear operator $F$ on
$\mathbbm{C}[[x]]$ such that:
\begin{enumerateroman}
  \item $\psi (x) = (F.x_1, \ldots, F.x_{\nu})$ is a local formal, invertible
  diffeomorphism of $\mathbbm{C}^{\nu}$
  
  \item For any series $A, B$ in $\mathbbm{C}[[x]]$, $F. (\tmop{AB}) = (F.A)
  (F.B)$
\end{enumerateroman}
is the substitution automorphism associated to $\psi$.

The link with Lie algebras can be put this way (see {\cite{Ilya}}):

\begin{prop}
  Let $\mathcal{D}^{}$ (resp. $\mathcal{D}_{\geqslant 1}$) the set of
  derivations $\sum_{i = 1}^{\nu} f_i (x) \frac{\partial}{\partial x_i} = X
  (x)$ such that $f (0) = 0$ (resp. $f$ has no linear part). $\mathcal{D}^{}$
  (resp. $\mathcal{D}_{\geqslant 1}$) is a complete graded Lie algebra (for
  the Lie bracket associated to the composition of operators). Moreover, if
  $\mathcal{G}$ (resp. $\mathcal{G}_{\geqslant 1}$) is the group of
  substitution automorphisms associated to local formal, invertible
  diffeomorphisms (resp. identity-tangent formal diffeomorphisms), then the
  exponential map:
  \[ \exp (X) = \tmop{Id}_{\mathbbm{C}[[x]]} + \sum_{s \geqslant 1}
     \frac{1}{s!} X^s_{} \]
  defines an injective (resp. bijective) map from $\mathcal{D}^{}$ (resp.
  $\mathcal{D}_{\geqslant 1}$) to $\mathcal{G}$ (resp. $\mathcal{G}_{\geqslant
  1}$) whose inverse on $\mathcal{G}_{\geqslant 1}$ is given by the logarithm:
  \[ \log (F) = \sum_{s \geqslant 1} \frac{(- 1)^{s - 1}}{s} (F -
     \tmop{Id}_{\mathbbm{C}[[x]]}) \]
\end{prop}

A proof can be found in {\cite{Ilya}} and we refer to \cite{MenCM} and
\cite{MenDE} for some applications of this operator formalism in dimension
$\nu = 1$. There is no difficulty in proving such a proposition. Let us just
notice that the graduation is related to the action of such derivation on
monomials : A vector field in $\mathcal{D}^{}$ (resp. $\mathcal{D}_{\geqslant
1}$) is given by a series of operators such as $c x_1^{n_1} \ldots
x_{\nu}^{n_{\nu}} \frac{\partial}{\partial x_i}$ with $n_1 + \ldots + n_{\nu}
\geqslant 1$ (resp. $n_1 + \ldots + n_{\nu} \geqslant 2$) that acts on
monomials  $x_1^{m_1} \ldots x_{\nu}^{m_{\nu}}$:
\[ \left( \left. c x_1^{n_1} \ldots x_{\nu}^{n_{\nu}} \frac{\partial}{\partial
   x_i} \right) . x_1^{m_1} \ldots x_{\nu}^{m_{\nu}} = x_1^{m_1 + n_1} \ldots
   x_i^{m_i + n_i - 1} \ldots x_{\nu}^{m_{\nu} + n_{\nu}} \right. \]
so that the total degree goes from $m_1 + \ldots + m_{\nu}$ to $m_1 + \ldots +
m_{\nu} + n_1 + \ldots + n_{\nu} - 1$ and the graduation for such an operator
is then $n_1 + \ldots + n_{\nu} - 1$.

Note also that this proposition implies that any substitution automorphism in
$\mathcal{G}_{\geqslant 1}$ is a differential operator. This could have been
proved directly, using the Taylor formula: If $\psi (x) = (x_1 + u_1 (x),
\ldots, x_{\nu} + u_{\nu} (x) = x + u (x)$, then
\[ F_{\psi} .A (x) = A (\psi (x)) = A (x + u (x)) = A (x) + \sum_{s \geq 1}
   \frac{1}{s!} \sum_{1 \leqslant i_1, \ldots, i_s \leqslant \nu} u_{i_1} (x)
   \ldots .u_{i_s} (x) \frac{\partial^s A (x)}{\partial x_{i_1} \ldots .
   \partial x_{i_s}} \]

If we look back at the (formal) flow of a vector field (or rather derivation)
$X$ in $\mathcal{D}^{}$ (resp. $\mathcal{D}_{\geqslant 1}$), the reader can
check that the associated one parameter family of substitution automorphisms
$F^t_X$ in $\mathcal{G}$ (resp. $\mathcal{G}_{\geqslant 1}$) is given by
$F^t_X = \exp (t X)$ and satisfies the operator equation
\[ \frac{\partial}{\partial t} . \text{$F^t_X$} = F^t_X .X \]
On the same way the homological equation can be formulated as follows : Let $X
= X_0 + B$ be a vector field of $\mathcal{D}$, with linear part $X_0$ and $B
\in \text{$\mathcal{D}_{\geqslant 1}$}$, if this vector field is conjugated by
$\psi$ to its linear part and $F = F_{\psi}$, then:
\[ F.X_0 = X.F = (X_0 + B) .F \]
and if $G$ is the inverse of $F$, this reads
\[ \tmop{ad}_{X_0} (G) = [X_0, G] = X_0 .G - G.X_0 = G.B \]
so that the homological equation is an equation relating an element of
$\mathcal{D}_{\geqslant 1}$ and an element of \ $\mathcal{G}_{\geqslant 1}$,
through the adjoint action of $X_0$: $B$ is a logarithmic derivative of $G$
for the derivation $d = \tmop{ad}_{X_0}$ (that respects the graduation). This is precisely the kind of equation we studied in the abstract context of
complete graded Lie algebras, equipped with a {\tmem{graded derivation}} $d$ \
that has essentially the same properties as in the case of the adjoint action
for vector fields. For example, if $X_0 = \sum_{i = 1}^{\nu} \lambda_i x_i
\frac{\partial}{\partial x_i}$, then
\[ d \left( \left. c x_1^{n_1} \ldots x_{\nu}^{n_{\nu}}
   \frac{\partial}{\partial x_i} \right) = \tmop{ad}_{X_0} \left( c x_1^{n_1}
   \ldots x_{\nu}^{n_{\nu}} \frac{\partial}{\partial x_i} \right) = (
   \left\langle \lambda, n \right\rangle - \lambda_i) c x_1^{n_1} \ldots
   x_{\nu}^{n_{\nu}} \frac{\partial}{\partial x_i} \right. \]
so that a vector field $X_0 + B$ is resonant if $d$ is not invertible on
$\mathcal{D}_{\geqslant 1}$. One can also notice in this case that, on
restriction to $\mathcal{D}_{\geqslant 1}$, we have $\mathcal{D}_{\geqslant 1}
= \ker d \oplus \tmop{Im} d$ and $\ker d$ is a sub-Lie algebra.

\bibliographystyle{plain}


\begin{thebibliography}{}

\end{thebibliography}


\begin{thebibliography}{10}

\bibitem{OB}
Olivier Bouillot.
\newblock {\em Invariants analytiques des diff{\'e}omorphismes et Multizetas}.
\newblock PhD thesis, Paris-Sud University, 2011.

\bibitem{bour}
Nicolas Bourbaki.
\newblock {\em Lie groups and {L}ie algebras. {C}hapters 1--3}.
\newblock Elements of Mathematics (Berlin). Springer-Verlag, Berlin, 1998.
\newblock Translated from the French, Reprint of the 1989 English translation.

\bibitem{ck1}
Alain Connes and Dirk Kreimer.
\newblock Renormalization in quantum field theory and the {R}iemann-{H}ilbert
  problem. {I}. {T}he {H}opf algebra structure of graphs and the main theorem.
\newblock {\em Comm. Math. Phys.}, 210(1):249--273, 2000.

\bibitem{ck2}
Alain Connes and Dirk Kreimer.
\newblock Renormalization in quantum field theory and the {R}iemann-{H}ilbert
  problem. {II}. {T}he {$\beta$}-function, diffeomorphisms and the
  renormalization group.
\newblock {\em Comm. Math. Phys.}, 216(1):215--241, 2001.

\bibitem{CM}
Alain Connes and Matilde Marcolli.
\newblock {\em Noncommutative geometry, quantum fields and motives}, volume~55
  of {\em American Mathematical Society Colloquium Publications}.
\newblock American Mathematical Society, Providence, RI, 2008.

\bibitem{EFGBP}
Kurusch Ebrahimi-Fard, Jos{\'e}~M. Gracia-Bond{\'{\i}}a, and Fr{\'e}d{\'e}ric
  Patras.
\newblock A {L}ie theoretic approach to renormalization.
\newblock {\em Comm. Math. Phys.}, 276(2):519--549, 2007.

\bibitem{EFGM}
Kurusch Ebrahimi-Fard, Li~Guo, and Dominique Manchon.
\newblock Birkhoff type decompositions and the {B}aker-{C}ampbell-{H}ausdorff
  recursion.
\newblock {\em Comm. Math. Phys.}, 267(3):821--845, 2006.

\bibitem{KMP}
Kurusch Ebrahimi-Fard, Dominique Manchon, and Fr{\'e}d{\'e}ric Patras.
\newblock A noncommutative {B}ohnenblust-{S}pitzer identity for {R}ota-{B}axter
  algebras solves {B}ogoliubov's recursion.
\newblock {\em J. Noncommut. Geom.}, 3(2):181--222, 2009.

\bibitem{EFP}
Kurusch Ebrahimi-Fard and Fr{\'e}d{\'e}ric Patras.
\newblock Exponential renormalization.
\newblock {\em Ann. Henri Poincar\'e}, 11(5):943--971, 2010.

\bibitem{SNAG}
Jean Ecalle.
\newblock Singularit\'es non abordables par la g\'eom\'etrie.
\newblock {\em Ann. Inst. Fourier (Grenoble)}, 42(1-2):73--164, 1992.

\bibitem{ERL}
Jean Ecalle.
\newblock Compensation of small denominators and ramified linearisation of
  local objects.
\newblock {\em Ast\'erisque}, (222):4, 135--199, 1994.
\newblock Complex analytic methods in dynamical systems (Rio de Janeiro, 1992).

\bibitem{EV2}
Jean Ecalle and Bruno Vallet.
\newblock Passive and active resonance. {N}on-linear resurgence and
  isoresurgent deformations.
\newblock In {\em The {S}tokes phenomenon and {H}ilbert's 16th problem
  ({G}roningen, 1995)}, pages 103--138. World Sci. Publ., River Edge, NJ, 1996.

\bibitem{EV1}
Jean Ecalle and Bruno Vallet.
\newblock Correction and linearization of resonant vector fields and
  diffeomorphisms.
\newblock {\em Math. Z.}, 229(2):249--318, 1998.

\bibitem{FM}
Fr{\'e}d{\'e}ric Fauvet and Fr{\'e}d{\'e}ric Menous.
\newblock Ecalle's arborification--coarborification transforms and
  connes--kreimer hopf algebra.
\newblock arXiv:1212.4740.

\bibitem{Gal}
Giovanni Gallavotti.
\newblock A criterion of integrability for perturbed nonresonant harmonic
  oscillators. ``{W}ick ordering'' of the perturbations in classical mechanics
  and invariance of the frequency spectrum.
\newblock {\em Comm. Math. Phys.}, 87(3):365--383, 1982/83.

\bibitem{Ilya}
Yulij Ilyashenko and Sergei Yakovenko.
\newblock {\em Lectures on analytic differential equations}, volume~86 of {\em
  Graduate Studies in Mathematics}.
\newblock American Mathematical Society, Providence, RI, 2008.

\bibitem{Man}
Dominique Manchon.
\newblock Bogota lectures on hopf algebras, from basics to applications to
  renormalization.
\newblock In {\em Comptes-rendus des Rencontres math\'ematiques de Glanon
  2001}.

\bibitem{Mart}
Jean Martinet.
\newblock Normalisation des champs de vecteurs holomorphes (d'apr\`es {A}.-{D}.
  {B}rjuno).
\newblock In {\em Bourbaki {S}eminar, {V}ol. 1980/81}, volume 901 of {\em
  Lecture Notes in Math.}, pages 55--70. Springer, Berlin, 1981.

\bibitem{PM}
F.~{Menous} and F.~{Patras}.
\newblock {Logarithmic Derivatives and Generalized Dynkin Operators}.
\newblock {\em ArXiv e-prints}, June 2012.

\bibitem{MenBD}
Fr{\'e}d{\'e}ric Menous.
\newblock On the stability of some groups of formal diffeomorphisms by the
  {B}irkhoff decomposition.
\newblock {\em Adv. Math.}, 216(1):1--28, 2007.

\bibitem{MenDE}
Fr{\'e}d{\'e}ric Menous.
\newblock Formal differential equations and renormalization.
\newblock In {\em Renormalization and {G}alois theories}, volume~15 of {\em
  IRMA Lect. Math. Theor. Phys.}, pages 229--246. Eur. Math. Soc., Z\"urich,
  2009.

\bibitem{MenCM}
Fr{\'e}d{\'e}ric Menous.
\newblock Formulas for the {C}onnes-{M}oscovici {H}opf algebra.
\newblock In {\em Combinatorics and physics}, volume 539 of {\em Contemp.
  Math.}, pages 269--285. Amer. Math. Soc., Providence, RI, 2011.

\end{thebibliography}

\end{document}